\documentclass[12pt]{article}

\usepackage{latexsym}

\usepackage{amsmath}
\usepackage{amsthm}
\usepackage{amsxtra}
\usepackage{bbm}

\usepackage{amsfonts,amssymb}
\usepackage{latexsym}

\newtheorem{Th}{Theorem}
\newtheorem{Prop}[Th]{Proposition}
\newtheorem{Lm}[Th]{Lemma}
\newtheorem{Co}[Th]{Corollary}
\newtheorem{Conj}[Th]{Conjecture}

\theoremstyle{definition}
\newtheorem{Def}[Th]{Definition}

\newcommand{\aut}{\mathrm{Aut}}
\newcommand{\supp}{\mathrm{supp}}

\newcommand{\tr}{\mathrm{tr}}
\newcommand{\st}{\mathrm{St}}

\newcommand{\mr}{\mathrm{r}}

\newcommand{\Sub}{\mathrm{Sub}}

\newcommand{\diag}{\mathrm{diag}}
\newcommand{\fix}{\mathrm{Fix}}
\newcommand{\sym}{\mathrm{Sym}}
\newcommand{\rist}{\mathrm{rist}}

\newcommand{\id}{\mathrm{Id}}

\newcommand{\eg}{\text{e.g.\,}}
\newcommand{\ie}{\text{i.e.}}

\begin{document}
\title{On diagonal actions of branch groups and the corresponding characters}
\author{ {\bf Artem Dudko}  \\
                    IMPAN, Warsaw, Poland  \\
          adudko@impan.pl \\
         {\bf Rostislav Grigorchuk} \\        Texas A\&M University, College Station, TX, USA  \\      grigorch@math.tamu.edu }


\maketitle

\begin{abstract} We introduce notions of absolutely non-free and perfectly non-free group actions and use them to study the associated unitary representations. We show that every weakly branch group acts absolutely non-freely on the boundary of the associated rooted tree. Using this result and the symmetrized diagonal actions we construct for every countable branch group infinitely many different ergodic  perfectly non-free actions, infinitely many II$_1$-factor representations, and infinitely many continuous ergodic invariant random subgroups.
\end{abstract}

\noindent{\small MSC: 20C15,  20E08,  37A05. }

\section{Introduction.}
Branch groups were introduced by the second author in \cite{Gr00}.
The class of branch groups contains many examples
of groups with remarkable properties. Among them is the group of intermediate
growth constructed in \cite{Gr80} and studied in \cite{Gr84} to answer Milnor's
question on growth and Day's question on existence of a non-elementary  amenable group.
Weakly branch groups form a more general class of groups acting on spherically homogeneous
rooted trees. It contains, for instance, the Basilica group, which is the first example of
 amenable but not subexponentially amenable group and is an iterated monodromy group of the map
  $f(z)=z^2-1$ \cite{Nekr}. Branch and weakly branch groups have many important applications in group theory, combinatorics, spectral theory, holomorphic dynamics, probability theory, \textit{etc}. This paper is dedicated to studying diagonal actions of weakly branch groups and associated characters and factor representations. 

Our results are related to the notion of non-free actions studied by A. Vershik in \cite{V11} and \cite{V12}. An action of a group $G$ on a Lebesgue space (also called standard probability space) is \emph{totally non-free} if the sigma-algebra generated by the sets $\fix(g)$ of fixed points of elements $g\in G$ is the whole sigma-algebra $\Sigma$ of measurable subsets of this space. An action of $G$ is called \emph{extremely non-free} if on a set of a full measure, different points have different stabilizers. In \cite{V11}, Theorem 8, Vershik showed that for countable groups these two notions coincide. It is known \cite{Grig11} that the action of a weakly branch group $G$ on the boundary $\partial T$ of the associated rooted tree $T$ equipped with the unique invariant measure $\mu$ is totally non-free. Here we show that it has a stronger non-freeness. Namely, the sets $\fix(g),g\in G$ approximate measurable subsets of $\partial T$ arbitrarily well in $\mu$-measure. We call actions satisfying this property \emph{absolutely non-free}. We also introduce the notion of \emph{perfect non-freeness} of actions, which is stronger than total non-freeness, but weaker than absolute non-freeness. As a corollary of our results, we obtain that an appropriate factor of the diagonal action of a countable weakly branch group $G$ on $(\partial T,\mu)^{\otimes n}$ for any $n\in\mathbb N$ is perfectly non-free. Another consequence is that every branch group has infinitely many pairwise non-isomorphic ergodic perfectly (and hence totally) non-free measure-preserving actions, and therefore infinitely many ergodic continuous invariant random subgroups. We apply the notion of a perfectly non-free action to study associated characters and factor representations.


By a character $\chi$ on a group $G$ we mean a normalized (by $\chi(e)=1$ for the unit $e$ of the group), positive semi-definite complex valued function on $G$ that is constant on conjugacy classes. Indecomposable characters are extreme points in the simplex of characters and are in one-to-one correspondence with quasi-equivalence classes of finite type factor representations of $G$. Group characters  and  factor  representations  is a  classical  topic  of  the  theory of unitary representations of  locally compact  groups  and  operator  algebras \cite{Dix69}\footnote{To avoid confusion we notice that in \cite{Dix69} indecomposability is included in the definition of a character and characters are allowed to have infinite values. In modern terminology group characters usually are assumed to have finite values. In some literature indecomposability is included in the definition of a character (see \eg \cite{Be07} and \cite{Ok}) and in some literature it is not (as in the present paper; see \eg \cite{PetThom} and \cite{V11}). }. There are examples of countable groups with uncountably many  (\cite{VK}, \cite{DN07 Wreath}) or countably many of indecomposable characters (\cite{Sk}, \cite{Be07}, \cite{DM13 AF}). On the other hand, there are infinite groups admitting only two indecomposable characters: the trivial and the regular one (\cite{Ov71}, \cite{Pet}, \cite{PetThom}, \cite{DM14 Thompson}).

Important examples of characters on a group $G$ are given by the functions of the form $\chi(g)=\mu(\fix(g)),g\in G$, where $G$ acts by measure-preserving transformations on a Lebesgue space $(X,\mu)$. In \cite{V11} Vershik suggested that for an ergodic action of a countable group the corresponding character is indecomposable if and only if the action is totally non-free. It turns out, the relation between total non-freeness of actions and indecomposability of characters is more complicated. More precisely, we show in the Appendix that there exist ergodic totally non-free actions producing decomposable characters and there exist free actions producing indecomposable characters. However, we show in Theorem \ref{ThIndecomp} that for any ergodic perfectly non-free action of a countable group, the corresponding character is indecomposable. As a consequence of our results, we obtain for any countable branch group an explicit countably infinite collection of indecomposable characters and hence an explicit countably infinite collection of factor representations of type II$_1$.

The paper is organized as follows. In Section \ref{SecPrelim}, we recall the basic facts about characters, factor representations, group actions on rooted trees and branch groups. Also, we discuss Vershik's notions of extremely non-free and totally non-free actions and their relation to invariant random subgroups. In Section \ref{SecMain}, we introduce the notions of absolutely non-free and perfectly non-free actions and formulate the main results. In Section \ref{SecNF}, we provide the proofs of Theorems \ref{ThIndecomp} and \ref{ThANFimpliesPNF} concerning non-free actions of arbitrary countable groups. Also, we give a detailed proof of Vershik's Theorem \ref{ThTNFCond} which is only sketched in \cite{V11}. In Section \ref{SecANF}, we show that every weakly branch group acts absolutely non-freely on the boundary of the associated rooted tree. Section \ref{SecErg} is devoted to the ergodic decomposition of the diagonal actions. The results obtained there are used in Section \ref{SecChar} to show that any branch group admits infinitely many indecomposible characters.


\section{Preliminaries.}\label{SecPrelim}
In this section we give necessary preliminaries on representation theory, groups acting on rooted trees and totally non-free actions, and formulate the main results.
\subsection{Factor representations and characters.}\label{SubsecReps}
We start by briefly recalling some important notions from the theory of operator algebras (see \cite{KR1} and
 \cite{KR2} for details).
 \begin{Def}\label{DefWstar} An algebra $\mathcal M$ of bounded operators acting on a Hilbert space is called a von Neumann (for short $W^*$-) algebra if it is closed in the weak operator topology.
A $W^*$-algebra $\mathcal M$ is called a factor if its center consists only of multiples of the identity operator. \end{Def}
\noindent Notice that each von Neumann algebra on a separable Hilbert space can be written as a direct integral of factors (see \eg \cite{KR2}, Theorems 14.2.2 and 14.2.3). In this paper we will be concerned only with finite type $W^*$-algebras acting on separable Hilbers spaces. Such algebras admit a faithful finite trace (see \eg \cite{KR2}).

A unitary representation $\pi$ of a group $G$ is called a  \emph{factor representation} if the $W^*$-algebra $\mathcal M_\pi$ generated by the operators $\pi(g),g\in G$, is a factor.
\begin{Def}
A {\it character} on a group $G$ is a function $\chi:G\rightarrow
\mathbb{C}$ satisfying the following properties:
\begin{itemize}
\item[(1)] $\chi(g_1g_2)=\chi(g_2g_1)$ for any $g_1,g_2\in G$;
\item[(2)] the matrix
$\left\{\chi\left(g_ig_j^{-1}\right)\right\}_{i,j=1}^n$ is
positive semi-definite for any integer $n\geq 1$ and any elements $g_1,\ldots,g_n\in G$;
\item[(3)] $\chi(e)=1$, where $e$ is the  identity element of $G$.
\end{itemize} A character $\chi$ is called {\it indecomposable} if it
cannot be represented in the form $\chi=\alpha
\chi_1+(1-\alpha)\chi_2$, where $0<\alpha<1$ and $\chi_1,\chi_2$ are
distinct characters.
\end{Def} The character given by $\chi(g)=1$ for all $g\in G$ is called the trivial character. The character $\chi(g)=\delta_{g,e}$, where delta stands for the Kronecker delta-symbol, is called the regular character. Notice that the trivial character is indecomposable for any group $G$. The regular character is indecomposable if and only if the group $G$ is ICC (all conjugacy classes except the class of the identity element are infinite).

Two unitary representations $\pi_1,\pi_2$ of $G$ are called \emph{quasi-equivalent} if there is a von Neumann  algebra isomorphism $\omega:\mathcal M_{\pi_1}\to\mathcal M_{\pi_2}$ such that
$$\omega(\pi_1(g))=\pi_2(g)\;\;\text{for all}\;\;g\in G.$$
The Gelfand-Naimark-Siegal (abbreviated "GNS") con\-struc\-tion associates to a character $\chi$ on $G$ a triple $(\pi,\mathcal{H},\xi)$,
 where $\pi$ is a unitary representation of $G$ acting on a Hilbert space $\mathcal{H}$ and $\xi$ is
 unit vector in $\mathcal{H}$ such that
 $$(\pi(g)\xi,\xi)=\chi(g)\;\;\text{for all}\;\;g\in G\;\;(\text{see \eg}\;\;\cite{DM13 AF}).$$ Moreover, $\xi$ is
  \emph{cyclic} (linear combinations of the vectors of the form $m\xi$, where $m\in\mathcal{M}_\pi$, are
  dense in $\mathcal{H}$) and separating ($A\xi=0$ for $A\in\mathcal{M}_\pi$ implies $A=0$)
  for $\mathcal{M}_\pi$. The function $$\tr(m)=(m\xi,\xi),\;\;\text{where}\;\;m\in\mathcal{M}_\pi,$$ is a faithful normalized ($\tr(\id)=1$) trace on $\mathcal{M}_\pi$. The following fact is well-known.
  \begin{Prop}\label{PropGNS} For a countable group $G$ the GNS-construction gives a bijection between the indecomposable characters of $G$ and the classes of quasy-equivalence of finite type factor representations of $G$.
  \end{Prop}
  \noindent For the reader's convenience we provide a sketch of the proof of Proposition \ref{PropGNS}.
  \begin{proof} Let $\chi$ be a character of a countable group $G$ and $(\pi,\mathcal H,\xi)$ be the corresponding GNS-construction. Then $\mathcal M_\pi$ is a finite type algebra since it admits a finite faithful trace $\tr(m)=(m\xi,\xi)$, $m\in\mathcal M_\pi$. Assume that $\chi$ is indecomposable but $\mathcal M_\pi$ is not a factor. Let $P$ be a projection from the center of $\mathcal M_\pi$ such that $P\neq 0$ and $P\neq\id$. Let
  \begin{align*}\alpha=\|P\xi\|^2,\;\xi_1=
  \frac{1}{\sqrt\alpha}P\xi,\;\xi_2=\frac{1}{\sqrt{1-\alpha}}(\id-P)\xi,\\ \chi_1(g)=(\pi(g)\xi_1,\xi_1),
  \;\chi_2(g)=(\pi(g)\xi_2,\xi_2),\;g\in G.\end{align*} Then $\chi_1,\chi_2$ are characters of $G$. Since $\chi=\alpha\chi_1+(1-\alpha\chi_2)$ by the indecomposability of $\chi$ we obtain that $\chi_1=\chi_2$. Using this one can show that the map $A$ given by
  $$A\pi(g)\xi_1=\pi(g)\xi_2,\;\;A\pi(g)\xi_2=\pi(g)\xi_1,\;\;g\in G$$ extends by linearity and continuity to a bounded linear operator on $\mathcal H$. Moreover, $A$ commutes with $\pi(G)$, therefore belongs to the commutant of $\mathcal M_\pi$. But clearly $A$ does not commute with $P$. This contradicts the fact that $P$ is from the center of $\mathcal M_\pi$. Thus, indecomposable characters give rise to finite type factor representations.

  On the other hand, theorem 4.5.2 from \cite{KR1} implies that up to quasi-equivalence every finite type factor representation arises as a GNS-construction for some character $\chi$. Since by Theorem 8.2.8 from \cite{KR2} every finite type factor admits a unique normalized trace, GNS-constructions corresponding to distinct indecomposable characters cannot produce quasi-equivalent factor representations.

  Finally, assume that $\chi$ is not indecomposable: $\chi=\alpha\chi_1+(1-\alpha)\chi_2$, where $0<\alpha<1$ and $\chi_1,\chi_2$ are distinct characters of $G$. Let $(\pi_i,\mathcal H_i,\xi_i)$ be the GNS-construction associated to $\chi_i$ and $\tr_i$ be the corresponding trace on $\mathcal M_{\pi_i}$, $i=1,2$.
  Then $\pi$ is unitary equivalent to $\pi_1\oplus \pi_2$ and $\mathcal M_\pi$ is a subalgebra in $\mathcal M_{\pi_1}\oplus\mathcal M_{\pi_2}$. One can check that for any $0<\beta<1$ the function
  $\tr_\beta=\beta\tr_1\oplus(1-\beta)\tr_2$ is a normalized trace on $\mathcal M_\pi$. Moreover, for distinct $\beta$ the corresponding traces are distinct. By the uniqueness of $\tr$ on a factor we obtain that $\mathcal M_\pi$ is not a factor.
  \end{proof}
\noindent  Notice, given a finite type factor representation $\pi$ of $G$ with normalized trace $\tr$ on $\mathcal M_\pi$ the associated indecomposable character is given by $$\chi(g)=\tr(\pi(g)),g\in G.$$

One important class of examples of characters arises from group actions. Namely, let $(X,\mu)$ be a probability space with a measure preserving action of a group $G$ on it. Then the function $$\chi(g)=\mu(\fix(g)),\;\;\text{where}\;\;\fix(g)=\{x\in X:gx=x\},$$ is a character. The corresponding representation can be constructed in the following way. Denote by $\mathcal{\mathcal{R}}$  the orbit equivalence relation on $X$. For a measurable subset $A\subset X^2$ and $x\in X$ set $A_x=A
 \cap (X\times\{x\})$. Introduce a measure $\nu$ on $\mathcal{\mathcal{R}}\subset X^2$ by
 $$\nu(A)=\int\limits_X |A_x|d\mu(x).$$ Notice
 that also $\nu(A)=\int\limits_X |A^y|d\mu(y)$, where $A^y=A
 \cap (\{y\}\times X)$. 
 \begin{Def}\label{DefGruppoid} The (left) groupoid representation of $G$ is the representation $\pi$
  in $L^2(\mathcal{\mathcal{R}},\nu)$ defined by
$$(\pi(g)f)((x,y))=f(g^{-1}x,y).$$
\end{Def}\noindent Denote by $\xi$ the unit vector
$\xi(x,y)=\delta_{x,y}\in L^2(\mathcal{\mathcal{R}},\nu)$, where $\delta_{x,y}$ is the Kronecker delta-symbol. It is straightforward to verify that
$$(\pi(g)\xi,\xi)=\mu(\fix(g))=\chi(g).$$ Moreover, the restriction of the representation $\pi$ on the cyclic hull of the vector $\xi$ is unitarily equivalent to the GNS-representation corresponding to $\chi$.

Denote by $\mathcal{M}_{\pi}$ the $W^*$-algebra generated by the operators of the representation $\pi$. In this paper we are also interested in another algebra. Namely, for a function $m\in L^\infty(X,\mu)$ introduce operators $m_l:L^2(\mathcal{R},\nu)\to L^2(\mathcal{R},\nu)$ by
$$(m_lf)(x,y)=m(x)f(x,y).$$
Denote by $\mathcal{M}_{\mathcal{R}}$ the $W^*$-algebra generated by $\mathcal{M}_{\pi}$
 and the operators $m_l,m\in L^2(X,\mu)$. This algebra is sometimes referred to as a Murray-von Neumann or Krieger algebra.
Observe that for an ergodic action of a group $G$ the algebra $\mathcal{M}_{\mathcal{R}}$ is a factor of finite type
  (see \cite{FM2}, Proposition 2.9(2)) and the vector $\xi$ is cyclic and separating for $M_{\mathcal R}$ (see \cite{FM2}, Proposition 2.5).

\subsection{Groups acting on rooted trees.}
In this subsection we give all necessary preliminaries on groups acting on rooted trees. We refer the reader to \cite{Grig11}, \cite{GNS00} and  \cite{Nekr} for the details. Here we will focus on the case of regular rooted trees. Some of the results of the present paper are true for a more general class of spherically homogeneous rooted trees, however, the proofs become more technical.

  Let $\mathcal A_d$ be a finite alphabet and $d\geqslant 2$ be the number of elements of $\mathcal A_d$. The vertex set of the regular rooted tree $T=T_d$ determined by $A_d$ is the set of all finite words over $\mathcal A_d$. Two vertices $v$ and $w$ are connected by an edge if one can be obtained from another by adding or erasing one letter from $\mathcal A_d$ at the end of the word. The root of the tree is represented by the empty word.
\begin{Def} The automorphism group $\aut(T)$ is the group of all graph isomorphisms of the tree $T$ onto itself.
\end{Def} Notice that elements of $\aut(T)$ preserve the root and the levels of $T$. Denote by $V_n$ the set of vertices of the $n$-th level of $T$, that is the set of words of lengths $n$ over $\mathcal A_d$. For a vertex $v$ of $T$ denote by $T_v$ the subtree of $T$ with the root vertex $v$. By definition, the boundary $X=\partial T$ of $T$ is the set of all infinite paths without backtracking starting from the root. The set $X$ can be naturally identified with the set of all infinite words over $\mathcal A_d$ supplied with the product topology. For $v\in V_n$ let $X_v$ be the set of all paths from $X$ passing through $v$. Equip $X$ with the uniform Bernoulli measure, so that $\mu(X_v)=\tfrac{1}{d^n}$ for each $v\in V_n$. The measure $\mu$ is $\aut(T)$-invariant.

Let $G<\aut(T)$. For $n\in\mathbb N$ and a vertex $v\in V_n$ denote by $\st_G(v)$ the subgroup of all elements which fix $v$ and by $\st_G(n)$ the subgroup of all elements which fix each vertex of $V_n$:
$$\st_G(v)=\{g\in G:gv=v\},\;\;\st_G(n)=\bigcap\limits_{v\in V_n}\st_G(v). $$ Observe that $\st_G(n),n\in\mathbb N$ are normal subgroups of finite index in $G$.
 By definition, the rigid stabilizer $\rist_G(v)$ of a vertex $v$ is the subgroup of elements $g\in G$ which act trivially on the complement of the subtree $T_v$. The rigid stabilizer of $n$-th level is defined as the subgroup generated by the rigid stabilizers of the vertices $v$ from $V_n$, \ie $$\rist_G(n)=\langle\rist_G(v):v\in V_n\rangle.$$
\begin{Def}\label{DefBr} A group $G<\aut(T)$ is called \emph{branch} if it acts transitively on each level $V_n$ of $T$ and for each $n$ $\rist_G(n)$ is a finite index subgroup in $G$. A group $G<\aut(T)$ is called \emph{weakly branch} if it acts transitively on each level $V_n$ of $T$ and $\rist_G(n)$ is infinite for each $n$ (equivalently, $\rist_G(v)$ is nontrivial for each vertex $v$).
\end{Def}
Every branch group is weakly branch. From Proposition 6.5 in \cite{GNS00} it follows that $(G,\partial T,\mu)$ is ergodic iff the action is level transitive. The same holds for the properties of minimality, topological transitivity and unique ergodicity. In particular, weakly branch groups act on $(\partial T,\mu)$ ergodically.

\subsection{Vershik's non-free actions.}

 Let $(X,\Sigma,\mu)$ be a Lebesgue space (see \cite{Rokh} or  \cite{Glas03} for the definition and properties of a Lebesgue space). Neglecting subsets of measure $0$ without loss of generality we will assume that
 $(X,\Sigma,\mu)$ is
isomorphic to an interval with the Lebesgue measure, a finite or countable set of atoms, or a disjoint
union of both.
A countable collection
$\mathcal{F}$ of measurable subsets of $X$ \emph{separates} points of $X$ if
for all $x\neq y\in X$ there exists $A\in \mathcal{F}$ such that $x\in A, y\notin A$ or $y\in A, x\notin A$.
By definition, the sigma-algebra
generated by a collection of measurable sets $\mathcal{F}$ and by the measure $\mu$ consists of all
measurable sets $B\subset X$
for which there exists a set $A$ in the Borel sigma-algebra generated by $\mathcal{F}$
 such that $\mu(B\Delta A)=0$ (see \cite{Rokh}). There is a countable collection $\mathcal{F}$ of measurable subsets of $X$ separating points of $X$.   By the theorem on bases (see \cite{Rokh}, p. 22), $\mathcal{F}$ is a basis, that is in addition to separability, it
  generates $\Sigma$.

For a measurable automorphism $g$ of $X$ denote by
 $\fix(g)$ the set of fixed points of $g$: $\fix(g)=\{x\in X:gx=x\}.$ Notice that $\fix(g)$ is defined up to a set of zero measure. 
%
In \cite{V11} Vershik introduced two notions associated with non-free actions:
\begin{Def}\label{DefTotNonfree} A measure preserving action of a countable group $G$ on a Lebesgue space
$(X,\Sigma,\mu)$ is called \emph{totally non-free} if the collection
of sets $\fix(g),g\in G$ and sets of zero measure generates the sigma-algebra $\Sigma$.
\end{Def}
\begin{Def}\label{DefExtrNonfree} A measure preserving action of a group $G$ on a Lebesgue space
$(X,\Sigma,\mu)$ is called \emph{extremely non-free} if there exists $A\subset X,\mu(A)=1$ such that for each
$x,y\in A,x\neq y$ one has $\st_G(x)\neq \st_G(y)$.
\end{Def}
\noindent Vershik in \cite{V11} showed that for a countable group $G$ these two notions coincide:
\begin{Th}\label{ThTNFCond}
 For a countable group $G$ an action on a Lebesgue space is totally non-free if and only
if it is extremely non-free.\end{Th}
\noindent Since some of the details of the proof of Theorem \ref{ThTNFCond} are only sketched in \cite{V11} here we provide a complete proof of Theorem \ref{ThTNFCond}.

In \cite{V11}, Vershik suggested the following:
\begin{Conj}\label{ClVer} For an ergodic action of a countable group $G$ on a Lebesgue space $(X,\mu)$ the associated character $$\chi(g)=\mu(\fix(g))$$ is indecomposable if and only if the action is totally non-free.
\end{Conj}
\noindent
However, as we show in the appendix, both directions of the above conjecture are false. In particular, the fact that the action is ergodic and totally non-free does not imply that the corresponding character is indecomposable. Here we show that indecomposability of the character follows from a stronger notion of non-freeness which we introduce in Section \ref{SecMain}.

 Extremely non-free actions are useful for constructing continuous ergodic \emph{invariant random subgroups}. Consider the space of subgroups $\Sub(G)$ of a group $G$ equipped with the topology generated by the sets of the form
$$U_{h_1,\ldots,h_k,g_1,\ldots,g_m}=\{H\in\Sub(G):h_i\in H\text{ for every }i,\;g_j\notin H\text{ for every }i\},$$ where $k,m\in\mathbb N$ and $h_1,\ldots,h_k,g_1,\ldots,g_m\in G$. The group $G$ acts on this space by conjugation:
$$H\in\Sub(G)\to g(H)=gHg^{-1}\in\Sub(G).$$ An invariant random subgroup (IRS) of $G$ is a $G$-invariant probability measure on $\Sub(G)$. An IRS is called continuous if it has no atoms.

Given an action of a group $G$ on a Lebesgue space $(X,\mu)$ consider the map \begin{equation*}\label{EqIRS}\Psi:X\to\Sub(G),\;\;\Psi(x)=\st_G(x).\end{equation*} If the action is extremely non-free then $\Psi$ is an isomorphism modulo sets of zero measure between $(X,\mu)$ and $(\Sub(G),\Psi_*\mu)$ (see \cite{V11}), where $\Psi_*\mu$ denotes the pushforward measure. If, in addition, the action is ergodic and measure-preserving, and $\mu$ has no atoms, then $\Psi_*\mu$ is a continuous ergodic IRS of $G$.

Using the above construction, in \cite{V12} Vershik obtained a complete list of ergodic IRS of the
infinite symmetric group. In \cite{Bowen} Bowen constructed a continuum of ergodic IRS for each
nonabelian free group. In the present paper for any branch group $G$ we obtain a countable collection
 of continuous ergodic IRS of $G$.
\section{Main results.}\label{SecMain}
To study characters and groupoid representations associated to actions of countable groups we introduce two new notions of non-freeness.
\begin{Def}\label{DefANF}
 An action of a group $G$ on a measure space $(X,\Sigma,\mu)$ is called absolutely non-free if
for every measurable set $A$ and every $\epsilon>0$ there exists $g\in G$ such that
$\mu(\fix(g)\Delta A)<\epsilon$.
\end{Def}
\noindent Instead of considering the sets $\fix(g)$ sometimes it is more convenient to work with their complements $\supp(g)$. For a measure-preserving automorphism $\phi$ of a Lebesgue space $(X,\Sigma,\mu)$ set $\supp(\phi)=\{x\in X:\phi(x)\neq x\}$. Notice that $\supp(\phi)$ is defined modulo $\mu$-zero measure and $\supp(\phi)=X\setminus\fix(\phi)$.
For a measure-preserving action of a group $G$ on a Lebesgue space $(X,\Sigma,\mu)$ and a measurable set $A\subset X$ introduce the subgroups $G_A<G$ of elements acting trivially outside $A$:
$$G_A=\{g\in G:\mu(\supp(g)\setminus A)=0\}.$$
\begin{Def}\label{DefPNF} Let a countable group $G$ act on a Lebesgue space $(X,\Sigma,\mu)$ by measure-preserving transformations. We will say that this action is \emph{perfectly non-free} if there exists a collection $\mathcal A$ of measurable subsets of $X$ such that $\mathcal A$ together with the sets of zero measure generates $\Sigma$ and for each $A\in\mathcal A$ the $G_A$-orbit $\{gx:g\in G_A\}\subset X$ is infinite for $\mu$-almost all $x\in A$.
\end{Def}
\noindent Notice that for an action of a countable group $G$ perfect non-freeness is a stronger condition than total non-freeness since each of the sets $A\in\mathcal A$ modulo zero measure can be obtained as a countable intersection of the sets of the form $\fix(g),g\in G$. We show that absolute non-freeness implies perfect non-freeness (see Theorem \ref{ThANFimpliesPNF}).

In Section \ref{SecNF} we prove the following.
\begin{Th}\label{ThIndecomp} Assume that the action of a countable group $G$ on a Lebesgue space $(X,\Sigma,\mu)$ is ergodic, measure-preserving and perfectly non-free. Let $\mathcal R$ be the corresponding equivalence relation on $X$ and $\pi$ be the associated groupoid representation. Then $\mathcal M_{\mathcal R}=\mathcal M_\pi$ and the corresponding character $\chi(g)=\mu(\fix(g)),g\in G,$ is indecomposable.
\end{Th}

Let $G$ act on a Lebesgue measure space $(X,\Sigma,\mu)$ by measure preserving transormations.
Consider the space $(X^n,\Sigma^n,\mu^n)$ (the n-th Cartesian power of $(X,\Sigma,\mu)$). Here $\Sigma^n$ is the sigma-algebra on $X^n$ generated by the sets of the form $A_1\times A_2\times\cdots\times A_n$, where $A_i\in \Sigma$.
Denote by $\diag_n$ the diagonal action of $G$ on $(X^n,\mu^n)$:
$$\diag_n(g)(x_1,\ldots,x_n)=(gx_1,\ldots,gx_n).$$ Notice that this action cannot be totally non-free, unless $\mu$
is concentrated in one point or $n=1$. Indeed,
$$\fix(\diag_n(g))=\fix(g)^n=\fix(g)\times\fix(g)\times\cdots\times\fix(g)\subset X^n$$ for all $g\in G$ and sets of this form cannot generate the sigma-algebra $\Sigma^n$ if $n>1$ and $\mu$ is not concentrated at one point.

Observe that for each $n\geqslant 1$ the action $\diag_n$ commutes with the action of the symmetric group $\sym(n)$ permuting factors of $X^n=X\times\cdots\times X$. Thus, this action
factors through the quotient $X_\sym^n:=X^n/\sym(n)$. The measure $\mu^n$ projects naturally
to an invariant measure $\mu^n_\sym$ on $X_\sym^n$. Denote by $\gamma_n:X^n\to X_\sym^n$ the quotient map
and by $\Sigma_n$ the sigma-algebra of measurable sets on $X_\sym^n$. Denote by $S_n(A):=\gamma_n^{-1}(\gamma_n(A))$ the
symmetrization of a subset $A\subset X^n$, that is the minimal subset of $X^n$ containing $A$ and invariant
under the action of $\sym(n)$.
\begin{Th}\label{ThANFimpliesPNF} If an action of a countable group $G$ on a Lebesgue space $(X,\Sigma,\mu)$ is absolutely non-free then for any $n$ the action $\diag_n$ of $G$ on $(X_\sym^n,\Sigma_\sym^n,\mu_\sym^n)$ is perfectly non-free.
\end{Th}
\noindent Notice that Theorem \ref{ThANFimpliesPNF} implies, in particular, that any absolutely non-free action of a countable group is also perfectly non-free.

  Assume that $G$ is a group acting on a regular rooted tree $T$, $X=\partial T$, $\mu$ is the uniform Bernoulli measure on $X$ and $\Sigma$ is the sigma-algebra of all measurable subsets of $X$. We prove the following:
\begin{Th}\label{ThDiagTNF} For any weakly branch group $G$ the action of $G$ on $(X,\Sigma,\mu)$ is absolutely non-free.
\end{Th}
  \noindent Recall that the action of a weakly branch group on the boundary of the corresponding rooted tree 
is ergodic with respect to the measure $\mu$. It is not hard to see that for $n\geqslant 2$ the action $\diag_n$ of $G$  on $(X_\sym^n,\Sigma_\sym^n,\mu_\sym^n)$ is not ergodic. Nevertheless, the following holds:
\begin{Prop}\label{ThDiagErg} For any countable branch group $G$ the action
 $\diag_n$ of $G$ on $(X_\sym^n,\Sigma_\sym^n,\mu_\sym^n)$ splits
 into countably many ergodic components of positive measures.
\end{Prop}
Further, let $G$ be a countable branch group. Denote by $\mathcal{E}_n$ the set of ergodic components of the action $\diag_n$ of $G$ on
$X^n_\sym$. For $\alpha\in\mathcal{E}_n$ let $X_\alpha$ be the corresponding subset of $X^n_\sym$. Denote by $\Sigma_\alpha$ the sigma-algebra of measurable subsets of $X_\alpha$ and by $\mu_\alpha$ the normalized restriction
of $\mu^n_\sym$ onto $X_\alpha$ (so that $\mu_\alpha(X_\alpha)=1$). Let $\diag_\alpha$ be the restriction of $\diag_n$ onto $X_\alpha$. Since by Theorem \ref{ThDiagTNF} the sets $\fix(\diag_n(g)),g\in G,$ generate the sigma-algebra $\Sigma_\sym^n$, the sets $\fix(\diag_\alpha(g)),g\in G$ for each $\alpha$ generate the sigma-algebra $\Sigma_\alpha$. Thus, the action $\diag_\alpha$ of $G$ on $(X_\alpha,\Sigma_\alpha,\mu_\alpha)$ is totally non-free. Introduce a function $\chi_\alpha$ on $G$ by
$$\chi_\alpha(g)=\mu_\alpha(\fix(\diag_n(g))\cap X_\alpha).$$ Let $\pi_\alpha$ be the groupoid representation of $G$ corresponding to the action $\diag_\alpha$. Using the map $\Psi_\alpha:X_\alpha\to \Sub(G)$ by $\Psi_\alpha(x)=\st_G(x)$ introduce an IRS $\phi_\alpha=(\Psi_\alpha)_*\mu$ on $G$.
Set $\mathcal E=\bigcup\limits_{n\in\mathbb N}\mathcal{E}_n$. From Theorems \ref{ThIndecomp}--\ref{ThDiagTNF} we obtain:
\begin{Co}\label{CoWBIndecomp}
The characters $\chi_\alpha$, $\alpha\in \mathcal{E}$, are indecomposable.
\end{Co}
\noindent We show
\begin{Prop}\label{ThDistChar}
The characters $\chi_\alpha$, $\alpha\in \mathcal{E}$, are pairwise distinct.
\end{Prop}
\noindent From Corollary \ref{CoWBIndecomp} and Proposition \ref{ThDistChar} we immediately obtain the following:
\begin{Th} Every branch group has infinitely many
pairwise non-isomorphic totally non-free actions $\diag_\alpha$, infinitely many indecomposable characters $\chi_\alpha$, infinitely many pairwise non-isomorphic finite type factor representations $\pi_\alpha$ and infinitely many continuous ergodic invariant random subgroups $\phi_\alpha$.
\end{Th}
\section{Non-free actions and the corresponding characters.}\label{SecNF}

\noindent As some details of proof of Theorem \ref{ThTNFCond} are only sketched in \cite{V11} we give here a proof of these Theorem in full details.
\begin{proof}[\textbf{Proof of Theorem \ref{ThTNFCond}}] Assume that an action of $G$ on $(X,\Sigma,\mu)$ is extremely non-free. Let $A$ be the set from the definition  of the extremely non-free action. Let $x,y\in A,x\neq y$. Then there exists $g\in G$ for which
$x\in \fix(g),y\notin \fix(g)$. Thus the collection of sets $\fix(g),g\in G$ separates points from $A$. By Theorem on bases from
\cite{Rokh} (see page 22) the collection of sets $\fix(g),g\in G$, is a basis for the Lebesgue space $(A,\Sigma\cap A,\mu|_A)$, that
is it generates the sigma-algebra of all measurable sets on $A$, and so on $X$ as well. Here $\Sigma\cap A=\{B\in\Sigma:B\subset A\}$ and $\mu|_A$ is the restriction of the measure $\mu$ on the subset $A$.

Further, assume that the action of $G$ on $(X,\Sigma,\mu)$ is totally non-free. Denote by $\Xi$
the partition on $X$ generated by the sets $\fix(g),g\in G$. The atoms of this partition are of the form
 $$\bigcap\limits_{g\in G}A_g,\;\;\text{where}\;\;A_g=\fix(g)\;\;\text{or}\;\;A_g=X\setminus \fix(g)
 \;\;\text{for each}\;\; g\in G.$$ In other words, $x,y$ belong to the same atom $\xi$ of the partition $\Xi$
if and only if $\st_G(x)=\st_G(y)$. Since $\fix(g),g\in G,$ generate the sigma-algebra of measurable sets
on $(X,\Sigma,\mu)$, each measurable set $B$ can be represented as $B_\Xi\cup B_0$, where $B_\Xi$ is a union of some atoms of the partition $\Xi$ and $\mu(B_0)=0$. Let  $B_i,i\in\mathbb{N},$ be a basis of measurable sets
on $(X,\Sigma,\mu)$. Let $B_i=B_{i,\Xi}\cup B_{i,0}$ be the corresponding subdivisions. Set $K=\cup_i B_{i,0}$.
By the definition of the basis, the sets $B_i$ separate
points of $X$. In particular, for each $x\neq y\in X\setminus K$ there exists $B_j$ such that $x\in B_j,y\notin B_j$. Then $x\in B_{j,\Xi},y\notin B_{j,\Xi}$. It follows that $x,y$ are not in the same atom of $\Xi$ and $\st_G(x)\neq\st_G(y)$, which finishes the proof.
\end{proof}\noindent
The following fact is folklore:
 \begin{Lm}\label{LmFixed} Let $\kappa$ be a unitary representation of a discrete group $\Gamma$ on a Hilbert space $H$. Set $H_1=\{\eta\in H:\kappa(g)\eta=\eta\;\;\text{for all}\;\;g\in \Gamma\}$. Then the orthogonal projection $P$ onto $H_1$ belongs to $\mathcal M_\kappa$.
 \end{Lm}
 \begin{proof} Let $B\in\mathcal M_\kappa'$. Then $$\kappa(g)B\eta=B\kappa(g)\eta=B\eta$$ for every $\eta\in H_1,g\in \Gamma$. This implies that $BH_1\subset H_1$ and so $BP=PBP$. Same argument shows that $B^{*}P=PB^{*}P$, where $^{*}$ stands for the operation of conjugation in $B(H)$. Conjugating the latter identity we obtain that $PB=PBP=BP$. By von Neumann Bicommutant Theorem (see \eg Theorem 5.3.2 in \cite{KR1}) we get that $P\in \mathcal (M_\kappa')'=\mathcal M_\kappa$.
 \end{proof}
\begin{proof}[\textbf{Proof of Theorem \ref{ThIndecomp}}]
Let $A\in\mathcal A$. Consider the space
$$\mathcal L_{A}=\{\eta\in L^2(\mathcal R,\nu):\nu(\supp(\eta)\setminus (A\times X))=0\}.$$ Let us show that $\mathcal L_{A}$ coincides with the subspace of $L^2(\mathcal R,\nu)$ consisting of $\pi(G_{X\setminus A})$-invariant vectors. Clearly, $\pi(G_{X\setminus A})$ acts trivially on $\mathcal L_{A}$. It remains to show that there are no non-zero $\pi(G_{X\setminus A})$-invariant vectors in $\mathcal L_{X\setminus A}$. Assume that there exists such non-zero vector $\eta$. By the definition of perfect non-freeness  for $\mu$-almost every $x\in X\setminus A$ the orbit $G_{X\setminus A}x$ is infinite. It follows that for $\nu$-almost all $(x,y)\in \supp(\eta)$ there exists infinitely many $z$ such that $\eta(z,y)=\eta(x,y)$. By the definition of the groupoid construction we obtain a contradiction ($\eta$ has an infinite norm).

Thus, from Lemma \ref{LmFixed} we obtain that the orthogonal projection onto the subspace
$\mathcal L_{A}$ belongs to $\mathcal M_\pi$. Notice that this projection coincides with the operator of multiplication by the characteristic function  $\mathbbm{1}_{A}(x)$ of $A$ in $x$-coordinate. Since the sets from $\mathcal A$ generate $\Sigma$ we obtain that $\mathcal M_\pi\supset L^\infty(X,\mu)_l$ and so $\mathcal M_\pi=\mathcal M_{\mathcal R}$. Since the action of $G$ on $(X,\mu)$ is ergodic the algebra $\mathcal M_{\mathcal R}$ is a factor. Since $\pi$ is isomorphic to the GNS-construction associated to $\chi$ and factor representations are in a bijection with indecomposable characters via GNS-construction (see Subsection \ref{SubsecReps}), we obtain that $\chi$ is indecomposable.
\end{proof}

\begin{Lm}\label{LmANFLongOrb} Assume that an action of a countable group $G$ on a Lebesgue space $(X,\Sigma,\mu)$ is absolutely non-free. Then for any $n\in \mathbb N$ for any measurable set $A\subset X$ and any $\epsilon>0$ there exists measurable sets $B\supset A\supset C$ such that \begin{equation}\label{EqBAC}\mu(B\setminus C)<\epsilon\;\;\text{and}\;\;|\{gx:g\in G_B\}|\geqslant n\;\;\text{for}\;\;\mu{-a.e.}\;\;x\in C.\end{equation}
\end{Lm}
\begin{proof} We prove the lemma by induction on $n$. For $n=1$ the statement follows from the definition of absolute non-freeness. Assume that the statement is true for  $n=m$. Let $A\subset X$ be a measurable set. In the case when $\mu(A)=0$ the lemma is trivial. Let $\mu(A)>0$. Fix $\epsilon>0$. By the definition of absolute non-freeness there exists $g\in G$ such that $\mu(\supp(g)\setminus A)<\epsilon/8$. Construct disjoint measurable subsets $A_1,\ldots, A_k\subset \supp(g)$ such that $g(A_i)\cap A_i=\varnothing$ for $1\leqslant i\leqslant k$ and
$$\mu(A_1\cup A_2\cup\cdots\cup A_k)>\mu(\supp(g))-\epsilon/8.$$
By the inductive assumption for each $1\leqslant i\leqslant k$ there exists $B_i\supset A_i\supset C_i$ such that
$$\mu(B_i\setminus C_i)<\frac{\epsilon}{8k}\;\;\text{and}\;\;|\{hx:h\in G_{B_i}\}|\geqslant m\;\;\text{for}\;\;\mu\text{-a.e.}\;\;x\in C_i.$$ Set $D_i=C_i\setminus g^{-1}(B_i\setminus A_i)$. Notice that $$\mu(B_i\setminus D_i)<\frac{\epsilon}{4k}\;\;\text{and}\;\;
g(D_i)\cap B_i=\varnothing.$$ Set $C=(\bigcup\limits_{i=1}^k D_i)\cap A, B=(\bigcup\limits_{i=1}^k B_i)\cup A\cup \supp(g)$. It is straightforward to verify that \eqref{EqBAC} is satisfied with these sets for $n=m+1$.
\end{proof}

\begin{proof}[\textbf{Proof of Theorem \ref{ThANFimpliesPNF}}] Let an action of a countable group $G$ on a Lebesgue space $(X,\Sigma,\mu)$ be absolutely non-free.
Let $A\subset X$ be a measurable subset. Fix $\epsilon>0$.
Using Lemma \ref{LmANFLongOrb} by induction we construct a sequence of triples of sets
$B_i\supset A_i\supset C_i,i\geqslant 0$, such that $A_0=A$ and for every $i\geqslant 0$ \begin{align*}
\;\;\mu(B_i\setminus C_i)<2^{-i-1}\epsilon,\;\;A_{i+1}=B_i,
\\ |\{gx:g\in G_{B_i}\}|\geqslant i\;\;\text{for}\;\;\mu\text{-a.e.}\;\;x\in C_i.\end{align*} Set $A_\epsilon=\bigcup_{i=0}^\infty A_i$. Then
\begin{equation}\label{EqAeps}\mu(A_\epsilon\setminus A)<\epsilon\;\;\text{and}\;\;\{gx:g \in G_{A_\epsilon}\}\;\;\text{is infinite for}\;\;\mu\text{-a.e.}\;\;x\in A_\epsilon.\end{equation} Notice that for any measurable set $B$ one has $G_{\alpha_n(B)}=G_B$, where $$\alpha_n(B)=X_\sym^n\setminus \gamma_n(X\setminus B)=\{\gamma_n(x):x\in X^n,x_i\in B\;\text{for at least one}\;i\}$$ and $G_{\alpha_n(B)}<G$ is the subgroup of elements $g\in G$ such that $\diag_n(g)$ is supported on $\alpha_n(B)$. It follows that
$$\{\diag_n(g)\bar x:g \in G_{\alpha_n(A_\epsilon)}\}\;\;\text{is infinite for}\;\;\mu_\sym^n\text{-a.e.}\;\;\bar x\in \alpha_n(A_\epsilon).$$

 Further, take a countable collection $\mathcal A_0$ of measurable subsets of $X$ which separates $n$-element subsets of $X$, \ie\, such that for any $\bar x\neq\bar y\subset X$ with $|\bar x|=|\bar y|=n$ there exists $A\in \mathcal A_0$ with $$\bar x\subset A,\bar y\not\subset A\;\;\text{or}\;\;\bar y\subset A,\bar x\not\subset A.$$  Then the sets $\alpha_n(A)$, $A\in \mathcal A_0$, separate $\mu_\sym^n$-almost all points of $X_\sym^n$. For each $A\in \mathcal A_0$ construct a set $A_{1/n}$ as in \eqref{EqAeps}. Then the sets
 $\alpha_n(A_{1/n})$, where $A\in\mathcal A_0$ and $n\in\mathbb N$, also separate $\mu_\sym^n$-almost all points of $X_\sym^n$. Theorem on Bases (see \cite{Rokh}) implies that this collection of sets together with sets of $\mu_\sym^n$-zero measure generates $\Sigma_\sym^n$ which shows that the action $\diag_n$ is perfectly non-free.
\end{proof}


\section{Absolute non-freeness of actions of weakly branch groups.}\label{SecANF}
In this section we obtain Theorem \ref{ThDiagTNF} from a more general Proposition \ref{PropBrAbsNonfree}. The proof is based on several statements.
First we prove a combinatorial lemma.
\begin{Lm}\label{LmOnSubs} Let $n\in\mathbb{N}$ and $H<\sym(n)$ be a subgroup
acting transitively on $\{1,2,\ldots,n\}$. Let $A$ be a subset
 of $\{1,\ldots,n\}$ such that for all $g\neq h,$ $g,h\in H$ one has
 $|g(A)\Delta h(A)|\leqslant |A|$, where $|A|$ is the cardinality of $A$. Then $|A|>n/2$.
\end{Lm}
\begin{proof} Let $k=|A|$. Then for all $h\neq g\in H$ one has
 $$|h(A)\cap g(A)|=\frac{1}{2}(|h(A)|+|g(A)|-|h(A)\Delta g(A)|)\geqslant k/2.$$
 For each $h\in H$ introduce a vector $\xi_h\in \mathbb{C}^n$ by the rule:
 $$\xi_h=(x_1,\ldots,x_n),\;\;\text{where}\;\;x_i=\left\{\begin{array}{ll}
 0,&\text{if}\;\;i\notin h(A),\\1,&\text{if}\;\;i\in h(A).
 \end{array}\right.$$ Let $(\cdot,\cdot)$ be the standard scalar product in $\mathbb C^n$ and $\|\cdot\|$ be the corresponding norm. Then $\|\xi_h\|^2=k$ and $(\xi_h,\xi_g)=|h(A)\cap g(A)|\geqslant k/2$ for all $g\neq h\in H$.
 It follows that
 $$\|\sum\limits_{h\in H}\xi_h\|^2\geqslant \tfrac{k}{2}m(m+1),\;\;\text{where}\;\;m=|H|.$$
 On the other hand, the group $H$ acts on $\mathbb{C}^n$ by permuting coordinates and  $g(\xi_h)=\xi_{gh}$ for all $g,h\in H$. Since $H$ acts transitively on $\{1,\ldots,n\}$ and the vector
  $\sum\limits_{h\in H}\xi_h$ is fixed by $H$, we obtain:
  $$\sum\limits_{h\in H}\xi_h=(\tfrac{km}{n},\tfrac{km}{n},\ldots,\tfrac{km}{n}),\;\;
  \|\sum\limits_{h\in H}\xi_h\|^2=\tfrac{k^2m^2}{n}.$$ It follows that $km\geqslant n\tfrac{m+1}{2}$ and hence
  $k>\tfrac{n}{2}$.
\end{proof}

 \noindent We also need the following:
 \begin{Lm}\label{LmWeakBrProp} Let $G$ be a weakly branch group acting on a $d$-regular rooted tree $T$. Then for every vertex $v$ there exists $g\in G$ such that $\supp(g)\subset X_v$ and $\mu(\supp(g))\geqslant \tfrac{1}{d}\mu(X_v)$.
 \end{Lm}
 \begin{proof} Denote by $e$ the identity element of $G$. For an element $h\in G$ let $l(h)=\max\{l:h\in\st_G(l)\}.$
 Fix a vertex $v$ of the tree. Since $G$ is weakly branch there exists $g\neq  e$ with
 $\supp(g)\subset X_v$. Denote
 $$L=\min\{l(g):g\in G,g\neq  e,\supp(g)\subset X_v\}.$$
 For $h\in\aut(T)$ denote by $\sigma_h$ the permutation induced by $h$ on the set $V_L$ of vertices of level $L$ in $T$.  For a vertex $v\in V_n,n\in\mathbb N$, and $l\geqslant n$ set $V_l(v)=T_v\cap V_l$. For $g\in G$ denote by $W(g)$ the set of vertices $w$ from
 $V_L(v)$ such that $g$ induces a nontrivial permutation on $V_{L+1}(w)$. Let
 $$k(g)=|W(g)|,\;\;K=\max\{k(g):g\in G,\supp(g)\subset X_v\}.$$ By the choice of $L$ we have that $K>0$. Fix an element $g\in G$ with $\supp(g)\subset X_v$
 such that $k(g)=K$.

Further, since $G$ acts transitively on $V_L$, we can find a number $m$
 and a collection of elements $H=\{h_1,h_2,\ldots,h_m\}\subset G$ such that the family $S=
 \{\sigma_{h_0},\sigma_{h_1},\ldots,\sigma_{h_m}\}$ of transformations of $V_L$
 forms a group preserving $V_L(v)$ and transitive on
 $V_L(v)$. Denote $g_i=h_igh_i^{-1}$. One has:
 $$W(g_i)=\sigma_{h_i}(W(g)),\;\;W(g_ig_j)\supset
 W(g_i)\Delta W(g_j)$$ for all $i,j$. It follows that the set $W(g)$ together with the group $S$ restricted
 to $V_L(v)$ satisfy the conditions of Lemma \ref{LmOnSubs}.
 Therefore, $K=|W(g)|> \tfrac{1}{2}|V_L(v)|$. For each $w\in W(g)\subset V_L(v)$ the element $g$
 induces a nontrivial permutation of $V_{L+1}(w)$, and thus
 $\supp(g)\supset X_{w_1}\cup X_{w_2}$ for some vertices $w_1,w_2\in V_{L+1}(w)$. It follows that $$\mu(\supp(g)\cap X_w)\geqslant \tfrac{2}{d}\mu(X_w)$$ for every $w\in W(g)$. This implies that $\mu(\supp(g))>
 \tfrac{1}{d}\mu(X_v)$.
\end{proof}

\begin{Prop}\label{PropBrAbsNonfree} Let $G$ be a weakly branch group, $T$
be the corresponding regular rooted tree, $X=\partial T$ and $\mu$ be the unique invariant measure on $X$.
Then for any clopen subset $A\subset X$ and any $\epsilon>0$
 there exists $g\in G$ such that $$\supp(g)\subset A\;\;\text{and}\;\;
 \mu(A\setminus \supp(g))<\epsilon.$$
 In particular, the action of $G$ on $(\partial T,\Sigma,\mu)$ is absolutely non-free.\end{Prop}
\begin{proof} Let $A\subset X$ be any clopen set and let $g_0= e$.
Construct by induction elements $g_n\in G,n=0,1,2\ldots$ such that
$\supp(g_n)\subset A$ and $\mu(A\setminus\supp(g_n))\leqslant \big(\tfrac{d}{d+1}\big)^n$.
 If $g_n$ is constructed choose vertices $v_1,\ldots v_k$ such that $X_{v_j}$ are disjoint subsets
 of $A\setminus\supp(g_n)$ and $$\sum\mu(X_{v_j})\geqslant \tfrac{d}{d+1}\mu(A\setminus\supp(g_n)).$$
  Using Lemma \ref{LmWeakBrProp} construct  elements $h_1,\ldots h_k$ such that
  $\supp(h_j)\subset X_{v_j}$ and $$\mu(\supp(h_j))\geqslant\tfrac{1}{d}\mu(X_{v_j}).$$
 Set $g_{n+1}=g_nh_1h_2\ldots h_k$. Then $\mu(A\setminus \supp(g_{n+1}))
 \leqslant \tfrac{d}{d+1}\mu(A\setminus \supp(g_n))$, which finishes the proof.
\end{proof}

\section{Ergodic components of the diagonal actions.}\label{SecErg}
In this section we prove {\bf Proposition \ref{ThDiagErg}}.

Let $G$ be a branch group acting on a regular rooted tree $T$, $X=\partial T$  and $\mu$ be the unique invariant probability measure on $X$. Fix $n\in\mathbb N$. First observe that since $(X_\sym^n,\Sigma_\sym^n,\mu_\sym^n)$ is a finite quotient of $(X^n,\Sigma^n,\mu^n)$ it is sufficient to show that the action $\diag_n$ of $G$ on $(X^n,\Sigma^n,\mu^n)$ admits at most countably many ergodic components. For a level $k\geqslant \log_d n$ and an n-tuple $\mathrm v=(v_1,\ldots,v_n)$ of distinct vertices from $V_k$ set
$$X_{\mathrm v}=X_{v_1}\times X_{v_2}\times\cdots\times X_{v_n}\subset X^n.$$
Observe that union of the subsets of the form $S_n(X_{\mathrm v})$ of $X^n$ for all levels $k$ and all n-tuples  $\mathrm v$ of $n$ distinct vertices from $V_k$ is of the full measure in $X^n$. Introduce stabilizer of $X_\mathrm v$ in $G$ by: $$G_{\mathrm v}=\{g\in G:\diag_n(g)X_\mathrm v=X_\mathrm v\}.$$ Clearly, to prove Proposition \ref{ThDiagErg} it is sufficient to show for each $k$ and each $\mathrm v$ that for the action of $G_\mathrm v$
 the space $X_\mathrm v$ splits into finitely many ergodic components.

 \begin{Lm}\label{LmRistErg} For each vertex $u$ of $T$ the action of $\rist_G(u)$ on $X_u$ has finitely many ergodic components.
 \end{Lm}
 \begin{proof} Let $u\in V_k$.
Since $G$ is branch, the subgroup $$\rist_G(k)=\prod\limits_{w\in V_k}\rist_G(w)$$ is of finite index in $G$. Let $m=|G/\rist_G(k)|$, $g_1,\ldots,g_m$ be representatives of the right cosets $G/\rist_G(k)$ and $A\subset X_u$ be a measurable subset of positive measure invariant under the action of $\rist_G(u)$. Then $A$ is invariant under the action of $\rist_G(k)$ and $$B=\bigcup\limits_{i=1,\ldots,m}g_iA$$ is invariant under the action of $G$. It follows that $\mu(B)=1$ and $\mu(A)\geqslant \tfrac{1}{m}$. This shows that
 the action of $\rist_G(u)$ on $X_u$ has at most $m$ ergodic components. \end{proof}
 Let $\mathrm v=(v_1,\ldots,v_n)$ be an $n$-tuple with pairwise distinct $v_i\in V_k$. From Lemma \ref{LmRistErg} we obtain that
 the restriction of the action $\diag_n$ to $\prod\limits_{i=1}^n\rist_G(v_i)$ admits finitely many ergodic components in $X_\mathrm v$. Since $$G_\mathrm v\supset\rist_G(k)\supset \prod\limits_{i=1}^n\rist_G(v_i)$$ we obtain that $X_\mathrm v$ has finitely many ergodic components with respect to the action $\diag_n$ of $G_\mathrm v$, which finishes the proof of Proposition \ref{ThDiagErg}.

 In the next statement we show that Proposition \ref{ThDiagErg} fails when replacing ``branch'' by ``weakly branch''.
\begin{Prop} There exists a weakly branch group $G$ acting on a regular rooted tree $T$  and $n\in\mathbb N$ such that the action $\diag_n$ of $G$ on $(\partial T^n,\Sigma^n,\mu^n)$ has no ergodic components of positive measure.\end{Prop}
\begin{proof}
 Consider the binary rooted tree $T_2$. For $v\in V_n$ denote by $\sigma_v$ the switch of the
 branches of $T_2$ emitting from $v$. Consider the group $G$ generated by all elements of the form:
 \begin{equation}\label{EqGen}1)\;\;\sigma_v,\;v\in V_l,\;\;l\;\;\text{is even};\;\;\;2)\;\;
 h_l=\prod\limits_{v\in V_l}\sigma_v,\;\;l\;\;\text{is odd}.\end{equation}

  Notice that $G$ is weakly branch. Indeed, any vertex $u\in V_n$, $n\in\mathbb N$, can be encoded by a finite sequence $u(1),u(2),\ldots,u(n)$, where $u(j)\in\{0,1\}$ for all $j$. Using the generators of $G$ we can change any element of the sequence $u(j)$ not affecting the other elements of the sequence. Thus, we can obtain any other sequence of $0$-s and $1$-s, which means that $G$ is spherically transitive. On the other hand, for any vertex $u$ and any vertex $v\in T_u$ from an even level one has $\sigma_v\in G$. Therefore, $\rist_G(u)$ is nonempty.

  Observe that $G$ is not branch. Indeed, for any element $g\in\rist_G(1)<\st_G(1)$, the shortest representation of $g$ in terms of the generators of $G$ do not contain generators of the form $2)$, and thus $G/\rist_G(1)$ is infinite.

 Fix an even $n$. Let us show that the action $\diag_n$ of $G$ on $(X^n,\Sigma^n,\mu^n)$, where $X=\partial T$,
  does not have an ergodic component $A$ such that $\mu^n(A)>0$. Fix $j$. For $v=(v_1,\ldots,v_n)\in V_{2j}^n$ set
  $$r_i(v)=\sum\limits_{p=1}^n v_p(2i) \;(\text{mod}\; 2),\;\;\mr(v)=(r_1(v),\ldots,r_j(v)).$$ It is not hard
  to show that $\mr(v)$ is invariant under the action $\diag_n$ of $G$. Indeed,  for any $l$ and any $u\in V_l$ the element $\sigma_u$ changes only the $l+1$-st coordinate of any vertex. Thus, $$(\sigma_u(v_p))(2i)=v_p(2i)\;\;\text{for all}\;\;l\neq 2i-1\;\;\text{and all}\;\;p,$$ which implies that $$r_i(\diag_n(\sigma _u)(v))=r_i(\diag_n(h_l(v)))=r_i(v)$$ whenever $l\neq 2i-1$. If $l=2i-1$, then $(h_l(v_p))(2i)=1-v_p(2i)$ for each $p$. Therefore, since $n$ is even, $r_i(h_l(v))=r_i(v)$.

  Let  $\mr=(r_1,\ldots,r_j)\in\{0,1\}^j$. Denote
  $$Y(\mr)=\bigcup\limits_{v\in V_{2j}^n,\mr(v)=\mr} X_v.$$ Then sets $Y(\mr)$ are invariant under the
  action of $\diag_n(G)$, form a partition of $X^n$ and are of measure $2^{-j}$ each. Thus, $X^n$
  can be split into invariant subsets of arbitrarily small measure and so cannot have ergodic components
  of positive measure.
\end{proof}
\section{Characters}\label{SecChar}
Fix a countable branch group $G$ acting on a regular rooted tree $T$. As before, let $X=\partial T$, let $\Sigma$ be the sigma-algebra of  all measurable subsets of $X$ and let $\mu$ be the unique invariant ergodic measure on $X$. Fix $n\in \mathbb{N}$ and consider the action $\diag_n$ of $G$ on  $X^n_\sym$. Recall that  $\mathcal{E}_n$ stands for the set of ergodic components of
 this action. By Proposition \ref{ThDiagErg}, $\mathcal E_n$ is countable. For $\alpha\in\mathcal{E}_n$ denote by $X_\alpha$ the corresponding ergodic component, by
 $\Sigma_\alpha$ the sigma-algebra of measurable subsets of $X_\alpha$ and by  $\mu_\alpha$ the normalized restriction
of the measure $\mu^n_\sym$ to $X_\alpha$ (so that $\mu_\alpha(X_\alpha)=1$). Let $\pi_\alpha$ be the corresponding groupoid representation of $G$ acting in $L^2(\mathcal R_\alpha,\nu_\alpha).$ As before denote by $\chi_\alpha$ the corresponding character of $G$:
$$\chi_\alpha(g)=\mu_\alpha(\fix(\diag_n(g))\cap X_\alpha).$$

\begin{proof}[\textbf{Proof of Proposition \ref{ThDistChar}}]
Let $\alpha\in\mathcal{E}_n$ and $\beta\in\mathcal{E}_m$ for some $n,m$.
Assume that $\chi_\alpha=\chi_\beta$. Then $$\mu_\alpha(\fix(\diag_n(g)))=\mu_\beta(\fix(\diag_m(g)))$$
 for all $g\in G$. Set $$Y_\alpha=\gamma_n^{-1}(X_\alpha),\;Y_\beta =
 \gamma_m^{-1}(X_\beta),\;c_\alpha=\mu^n(Y_\alpha),\;c_\beta=\mu^m(Y_\beta).$$
 By absolute non-freeness of the action of $G$ on $(X,\Sigma,\mu)$ (see Proposition \ref{PropBrAbsNonfree}) and the definition
 of the measures $\mu_\alpha,\mu_\beta$
 we obtain that \begin{equation}
 \label{EqMuAlBe}\tfrac{1}{c_\alpha}\mu^n(A^n\cap Y_\alpha)=\tfrac{1}{c_\beta}
 \mu^m(A^m\cap Y_\beta)\;\;\text{for all}\;\;A\in \Sigma.\end{equation}
 \vskip 0.3cm
\noindent{\bf Case 1:} $n=m$ and
$\alpha\neq\beta$. Introduce inductively collections $\mathcal B_k$ of measurable subsets of $X^n$ as follows.
 Set $\mathcal B_0=\{A^n:A\subset X\;\text{measurable}\}$. Denote by $\mathcal B_{k+1}$ the set of subsets
 of the form $B_1\sqcup B_2$, where $B_1,B_2$ are disjoint subsets from $\mathcal B_k$, or $B_1\setminus B_2$,
 where $B_1\supset B_2$ are subsets from $\mathcal B_k$. Clearly, $\mathcal B_{k+1}\supset \mathcal B_k$ for all $k$.
  From \eqref{EqMuAlBe}
 using induction it is easy to show that
 $$\tfrac{1}{c_\alpha}\mu^n(C\cap Y_\alpha)=\tfrac{1}{c_\beta}
 \mu^n(C\cap Y_\beta)\;\;\text{for any}\;\;k\;\;\text{and all}\;\;C\in\mathcal B_k.$$
 Let us show by induction that for any $k$ there exists $N(k)$ such that for any $B_1,B_2\in\mathcal B_k$ one has:
 $$B_1\cup B_2,B_1\cap B_2,B_1\setminus B_2\in \mathcal B_{N(k)}.$$
 Since $B_1\setminus B_2=
 B_1\setminus (B_1\cap B_2)$ and $B_1\cup B_2=(B_1\setminus B_2)\sqcup
 (B_1\cap B_2)\sqcup (B_2\setminus B_1)$ it is sufficient to consider the $B_1\cap B_2$ case only. \\
 \emph{Base of induction}, when $k=0$, is obvious, since $A_1^n\cap A_2^n=(A_1\cap A_2)^n$.\\
 \emph{Step of induction}. Assume that the statement is true for $k$. Let $B_1,B_2\in \mathcal B_{k+1}$.
  Then for $i=1,2$ there exists $B_{i1}, B_{i2}\in\mathcal B_k$ such that either
  $B_i=B_{i1}\sqcup B_{i2}$ or $B_i=B_{i1}\setminus
  B_{i2}$ and $B_{i1}\supset B_{i2}$. Assume for instance that
  $$B_1= B_{11}\sqcup B_{12}\;\;\text{and}\;\;B_2=B_{21}\setminus
  B_{22},\;\;\text{where}\;\;B_{21}\supset B_{22}.$$ Then
  $$B_1\cap B_2=((B_{11}\cap B_{21})\sqcup (B_{12}\cap B_{21}))\setminus
((B_{11}\cap B_{22})\sqcup (B_{12}\cap B_{22}))\in\mathcal B_{N(k)+2},$$
  since $B_{1i}\cap B_{2j}\in \mathcal B_{N(k)}$ for $i,j=1,2$.
  The other three cases can be treated similarly.

   Since the sets of the form $A^n,A\in\Sigma$ generate $\gamma_n^{-1}(\Sigma_\sym^n)$  we obtain that $\tfrac{1}{c_\alpha}\mu^n(C\cap Y_\alpha)=\tfrac{1}{c_\beta}
\mu^n(C\cap Y_\beta)$
 for all $C\in \gamma_n^{-1}(\Sigma_\sym^n)$. In particular, for $C=Y_\alpha$ we get: $\mu^n(Y_\alpha)=0.$ Obtained
  contradiction finishes the proof in the case $n=m$.

\noindent{\bf Case 2:} $n\neq m$. Without loss of generality let $n>m$. From \eqref{EqMuAlBe} and
   the arguments above it is not hard
  to show that for any $k$, any $C_1,C_2,\ldots,C_k\in\Sigma$ and any expression $F(B_1,B_2,\ldots,B_k)$ involving
  only operations of union, intersection and difference of the sets one has
  \begin{equation}\label{EqMuAlBe2}\tfrac{1}{c_\alpha}\mu^n(F(C_1^n,\ldots,C_k^n)\cap Y_\alpha)=\tfrac{1}{c_\beta}
  \mu^m(F(C_1^m,\ldots,C_k^m)\cap Y_\beta).\end{equation}
  Find pairwise disjoint sets $A_1,\ldots, A_n\in\Sigma$ such that
  $$\mu^n(\sym(A_1\times A_2\times \cdots\times A_n) \cap Y_\alpha)>0.$$ Set $k=n+1$,
   $C_{n+1}=\cup_{i=1}^n A_i$ and $C_j=C_{n+1}\setminus A_j$ for
   $j=1,\ldots,n$. Let $$F(B_1,\ldots,B_{n+1})=B_{n+1}\setminus \cup_{j=1}^n B_j.$$
   Then $F(C_1^n,\ldots,C_k^n)=\sym(A_1\times A_2\times \cdots\times A_n)$ and
   $F(C_1^m,\ldots,C_k^m)=\varnothing$. From \eqref{EqMuAlBe2} we obtain
   that $\mu^n(\sym(A_1\times A_2\times \cdots\times A_n)\cap Y_\alpha)=0$. This contradiction finishes  the proof.
\end{proof}

\section{Appendix.}\label{SecApp}
Here we show that both "if" and "only if" directions of Vershik's Conjecture \ref{ClVer} are false by providing corresponding counterexamples.
\begin{Prop} There exists an essentially free (and thus not totally non-free) ergodic measure-preserving action of a countable group $G$ on a Lebesgue space $(X,\mu)$ such that the associated character $\chi(g)=\mu(\fix(g)),g\in G$ is indecomposable.
\end{Prop}
\begin{proof}
First notice that if the action of a countable ICC (infinite conjugacy classes) group $G$ on a Lebesgue space $(X,\mu)$ is measure-preserving, ergodic and essentially free (that is $\mu(\fix(g))=0$ for all $g\neq e$) then the associated groupoid representation is regular and the associated character $\chi$ is the regular character (equal to $1$ at $e$ and zero at all other elements $g\in G$). It is known that a countable group is ICC if and only if the regular representation is a factor representation (see \cite{MuvN}, Lemma 5.3.4). It follows that $\chi$ is indecomposable.

To construct an example of such action take $X=U(n),n\in\mathbb N,n\geqslant 2$, the group of $n\times n$ unitary matrices. Since $U(n)$ is compact it admits a unique left Haar measure $\mu$ with $\mu(X)=1$. Let $\Gamma$ be a dense countable subgroup of $U(n)$ which does not contain any diagonal matrices except the identity. We leave as an exercise for the reader to check that $\Gamma$ is ICC and the action of $\Gamma$ on $(X,\mu)$ by left translations is measure-preserving, ergodic and essentially free.
\end{proof}
\begin{Prop} There exists an ergodic totally non-free action of a countable group $G$ on a Lebesgue space $(X,\mu)$ such that the corresponding character $\chi(g)=\mu(\fix(g)),g\in G,$ is not indecomposable.
\end{Prop}
\begin{proof} In the proof we introduce three group actions and corresponding groupoid constructions. For convenience, to distinguish the objects corresponding to different groups we put index of the group next to the corresponding object.

First, let $H$ be any group of permutations on a finite set $A$ (equipped with the uniform measure $u$) which is transitive and extremely non-free. For example, one can set $H$ to be the symmetric group $\sym(n)$ acting on $H=\{1,2,\ldots,n\}$. Let $(\pi_H,\mathcal H_H,\xi_H)$ be the corresponding groupoid construction. Notice that $\mathcal R_H=A\times A$ in this case. The vector
$$\eta=\frac{1}{|H|}\sum\limits_{h\in H}\pi_H(h)\xi_H$$ is a non-zero $\pi_H(H)$-invariant vector. Thus, $\pi_H$ contains a trivial component but is not trivial, therefore $\chi_H$ is not indecomposable. This provides the required example for the case of a finite space $X$.

To construct an example with $(X,\mu)$ non-atomic take any countable group $S$ with an ergodic totally non-free action on an non-atomic Lebesgue space $(Y,\lambda)$. Consider the product action of $G=H\times S$ on $(A\times Y,u\times\lambda)$. Observe that in this space $\fix(h,s)=\fix(h)\times \fix(s)$. Therefore, this action is totally non-free. Clearly, it is ergodic. Further, let $(\pi_G,\mathcal H_G,\xi_G)$ be the groupoid construction associated to the action of $G$. Then the subspace $\mathcal\eta\otimes L^2(\mathcal R_S,\nu_S)$ is $\pi_G(G)$-invariant and $\pi_G(H\times \{e_S\})$ acts trivially on it. Since $\pi_G|_{H\times \{e_S\}}$ is not trivial on cyclic hull of $\xi_G$ we obtain that $\chi_G$ is not indecomposable.
\end{proof}

\subsection*{Acknowledgement} The authors thank Anatoly Vershik and Tatiana Smirnova-Nagnibeda for
fruitful discussions at the preliminary stage of our investigation. We are grateful to
 Pierre de la Harpe and Bachir Bekka for careful reading of the first and revised versions of the paper,
 useful suggestions and critics. In particular, we are grateful to B. Bekka for the indication of the gap
 in our unsuccessful attempt to provide, in a preliminary version of the paper, a detailed proof
  of the "if" direction of  Vershik's Theorem 10 from \cite{V11} that lead us to finding counterexamples to it.
  The authors are thankful to Ken Dykema, David Kerr and Konstantin Medynets
   for useful remarks. 

 The authors acknowledge the support of  the  Swiss  National  Science Foundation. The second author acknowledges the support of the US National Science Foundation under the grant DMS-1207699 and of the US National Security Agency under the grant H98230-15-1-0328.



\begin{thebibliography}{99}

\bibitem{BG} L. Bartholdi and R. I. Grigorchuk, \emph{On the Spectrum of Hecke Type Operators Related to Some Fractal Groups,} Tr. Mat. Inst. im. V.A. Steklova, Ross. Akad. Nauk {\bf 231} (2000), 5-45  [Proc. Steklov Inst. Math. {\bf 231} (2000), 1-41.

\bibitem{BGS03} L. Bartholdi, R. I. Grigorchuk, and Z. \v{S}uni\'{c}, \emph{Branch Groups}, Handbook of Algebra (North-Holland, Amsterdam, 2003), Vol. 3, pp. 989-1112.

\bibitem{Be07} B. Bekka, \emph{Operator-algebraic superridigity for $SL_n(\mathbb Z)$, $n=3$}, Invent. Math. {\bf 169} (2007), no. 2, pp. 401-425.

\bibitem{Bowen} L. Bowen, \emph{Invariant random subgroups of the free group.} Groups Geom. Dyn. {\bf 9} (2015), no. 3, 891-916.


\bibitem{Dix69} J. Dixmier., \emph{Les C$^*$-alg\`ebres et leurs repr\'esentations}, Gauthier-
Villars 1969.

\bibitem{DM13 AF} A. Dudko and K. Medynets, \emph{On Characters of Inductive Limits of Symmetric Groups}, Journal of Functional Analysis, {\bf 264} (2013), no.7, 1565-1598.

\bibitem{DM14 Thompson} A. Dudko and K. Medynets, \emph{Finite Factor Representations of Higman-Thompson groups}, Groups, Geometry, and Dynamics, {\bf 8} (2014).

\bibitem{DN07 Wreath} A. Dudko and N. Nessonov, \emph{A description of characters on the infinite wreath product}, Methods Funct. Anal. Topology {\bf 13} (2007), no. 4, 301-317.

\bibitem{FM2} J. Feldman, C. Moore, \emph{Ergodic equivalence relations, cohomology, and von Neumann algebras. II,}
 Trans. Amer. Math. Soc., {\bf 234} (1977), no. 2, 325--359.

 \bibitem{Glas03} E. Glasner, \emph{Ergodic theory via joinings}, Math. Surv. and Mon., {\bf 101},
 Amer. Math. Soc. (2003).

 \bibitem{Gr80} R. I. Grigorchuk, \emph{Burnside's Problem on Periodic Groups,} Funkts. Anal. Prilozh. {\bf 14} (1), 53-54 (1980) [Funct. Anal. Appl. {\bf 14}, 41-43 (1980)].

 \bibitem{Gr84} R. I. Grigorchuk, \emph{Degrees of growth of finitely generated groups, and the theory of invariant means,} Izv.
Akad. Nauk SSSR, Ser. Mat. {\bf 48} (5), 939-985 (1984) [Math. USSR, Izv. {\bf 25}, 259-300 (1985)].

\bibitem{Gr00} R. I. Grigorchuk, \emph{Just Infinite Branch Groups,} New Horizons in Pro-p Groups (Birkhauser, Boston, MA,
2000), Prog. Math. {\bf 184}, pp. 121-179.

\bibitem{Grig11} R. I. Grigorchuk, \emph{Some topics in the dynamics of group actions on rooted trees},
Proceedings of the Steklov Institute of Mathematics,
Volume 273, Issue 1 (2011), pp 64-175.

\bibitem{GNS00} R. I. Grigorchuk, V. V. Nekrashevich, and V. I. Sushchanskii, \emph{Automata, Dynamical Systems, and Groups,} Tr.
Mat. Inst. im. V.A. Steklova, Ross. Akad. Nauk {\bf 231}, 134-214 (2000) [Proc. Steklov Inst. Math. {\bf 231}, 128-203
(2000)].

\bibitem{KR1} R.~Kadison and J.~Ringrose, \emph{Fundamentals of the theory of operator algebras. Vol. I. Elementary
theory.} Pure and Applied Mathematics, 100. Academic Press, Inc. [Harcourt Brace Jovanovich, Publishers], New York,
1983.

\bibitem{KR2} R.~Kadison and J.~Ringrose,  \emph{Fundamentals of the theory of operator algebras.
Vol. II. Advanced theory.} Pure and Applied Mathematics, 100. Academic Press, Inc., Orlando, FL, 1986.

\bibitem{VK} S. Kerov and A. Vershik, \emph{Characters and factor representations of the infinite symmetric group},
Soviet Math. Dokl., {\bf 23 }(1981) No.2, 389-392.

\bibitem{Mack} G. W. Mackey, \emph{ The Theory of Unitary Group Representations},
 Univ. Chicago Press, Chicago, 1976, Chicago
Lect. Math.

\bibitem{MuvN} F. Murray and J. von Neumann, \emph{On rings of operators}, IV. Ann. of
Math. (2) {\bf 44} (1943), 716-808.

\bibitem{Nekr} V. Nekrashevych, \emph{Self-similar Groups,} Am. Math. Soc., Providence, RI, 2005, Math. Surv. Monogr. {\bf 117}.

\bibitem{Ok} A.Okounkov, \emph{On representations of
 the infinite symmetric group}, Zap. Nauchn. Sem. S.-Peterburg. Otdel. Mat. Inst. Steklov. (POMI) {\bf 240}, 1997, Teor. Predst. Din. Sist. Komb. i Algoritm. Metody. 2, 166-228, 294; translation in
J. Math. Sci. (New York) {\bf 96} (1999), no. 5, 3550-3589.




\bibitem{Ov71} S. V. Ovchinnikov, \emph{Positive-definite functions on Chevalley groups}, Funkts. Anal. Prilozh., {\bf 5}:1 (1971), 91-92.

\bibitem{Pet} J. Peterson, \emph{Character rigidity for lattices in higher-rank groups}, preprint.

\bibitem{PetThom} J. Peterson and A. Thom,  \emph{Character rigidty for special linear groups},  J. Reine Angew. Math. {\bf 716} (2016), 207-228.

\bibitem{Rokh}  V. A. Rokhlin, \emph{On the basic ideas of  measure theory}, Mat. Sb. {\bf 25} (1) (1949), 107-150. Engl. transl.: \emph{On the fundamental ideas of measure theory}, (Am. Math. Soc., Providence, RI, 1952), AMS Transl., No. 71.

\bibitem{Sk}
 Skudlarek H. L., {\it Die unzerlegbaren Charaktere einiger
 diskreter Gruppen}, Math. Annal., {\bf 233} (1976), 213--231.


%
%

\bibitem{V11} A.~Vershik, {\it Nonfree Actions of Countable Groups and their Characters}.  (Russian)
Zapiski Nauchn. Semin. POMI \textbf{378} (2010), Teoriya Predstavleniĭ, Dinamicheskie Sistemy, Kombinatornye Metody. XVIII, 5-16.
English translation in J. Math. Sci. (N.Y.) {\bf 174} (2010), no. 1, 1-6

\bibitem{V12} A. Vershik, \emph{Totally non-free actions and the infinite symmetric group,} Mosc. Math. J., {\bf 12} (2012), no. 1, 193-212.

\end{thebibliography}
\end{document}